\documentclass[11pt]{article}
\title{A note on conjectures generalizing the road colouring theorem}
\author{Theo Morrison\\University of British Columbia, Vancouver, Canada\\ email: theogsm@student.ubc.ca}
\date{}

\usepackage{amsfonts, amsmath, amssymb, amsthm, bigints}
\usepackage[shortlabels]{enumitem}
\usepackage{tikz}
\usepackage{graphicx}
\usetikzlibrary{shapes.geometric}
\usetikzlibrary{arrows.meta}
\usetikzlibrary{arrows,decorations.pathmorphing}
\usetikzlibrary{positioning}
\tikzstyle{vtx} = [circle, minimum width = 2mm, fill, inner sep = 0pt]

\newtheorem{theorem}{Theorem}[section]

\newtheorem{definition}[theorem]{Definition}

\newtheorem{conjecture}[theorem]{Conjecture}

\newtheorem{proposition}[theorem]{Proposition}

\theoremstyle{definition}

\DeclareSymbolFont{symbolsC}{U}{txsyc}{m}{n}
\DeclareMathSymbol{\strictif}{\mathrel}{symbolsC}{74}
\DeclareMathSymbol{\strictfi}{\mathrel}{symbolsC}{75}
\DeclareMathSymbol{\strictiff}{\mathrel}{symbolsC}{76}
\usepackage{fullpage}

\begin{document}

\maketitle
\noindent
\section{Introduction}
The road colouring theorem characterizes the class of strongly connected directed graphs with constant out-degree that admit a synchronizing road colouring. The subject of this paper is a pair of related conjectures that generalize the road colouring theorem to graphs with non-constant out-degree. The first of these conjectures is the $O(G)$ conjecture, which was first proposed by Ashley-Marcus-Tuncel in \cite{AMT} as part of a characterization of isomorphism of one-sided markov chains. It was also observed in \cite{AMT} that a proof of the $O(G)$ conjecture would yield a proof of the (at the time unproven) road colouring theorem. The second conjecture, called the bunchy factor conjecture, strengthens the $O(G)$ conjecture (in some sense), and gives a more direct generalization of the road colouring theorem. In this paper we give reasons to believe that both conjectures are true. Our main results focus on two classes of graphs, proving the bunchy factor conjecture for both classes and the $O(G)$ conjecture for one. We also present computer simulations that give some empirical evidence for the conjectures.

\section{Background}
Throughout this paper, graphs are finite, directed, and allow loops and multiple edges. Let $G=(V(G),E(G))$ be a graph with vertex set $V(G)$ and edge set $E(G)$. For a path $w$ in $G$, we write $i(w)$ and $t(w)$ for the initial and terminal vertices of $w$. For a vertex $I\in V(G)$, we write $E_I(G)=\{e\in E(G): i(e)=I\}$ for the set of outgoing edges from $I$, $E^I(G)=\{e\in E(G): t(e)=I\}$ for the set of incoming edges to $I$, and $E_I^J(G)=E_I(G)\cap E^J(G)$ for the set of edges from $I$ to a vertex $J\in V(G)$. The follower states (terminal vertices of outgoing edges) of $I\in V(G)$ are denoted by $F(I)=t(E_I(G))$. We use $L(G)$ to denote the set of (finite) paths on $G$, and write $L_I(G)=\{w\in L(G): i(w)=I\}$ and $L^I(G)=\{w\in L(G): t(w)=I\}$ for the sets of paths that start and end at a vertex $I\in V(G)$. We say that $G$ is strongly connected if there is a directed path from $I$ to $J$ for all ordered pairs $(I,J)\in V(G)\times V(G)$. The period of $G$, denoted $\text{per}(G)$, is defined as the $\gcd$ of the lengths of the cycles in $G$, and we say that $G$ is aperiodic if $\text{per}(G)=1$. 

The ``road colourings" of Trahtman's road colouring theorem are edge colourings of a graph where each colour is assigned to exactly one outgoing edge from each vertex (ie. a colouring that is bijective when restricted to the outgoing edges of any vertex). A road colouring is said to be synchronizing if there is a sequence of colours such that every path coloured by that sequence ends at the same vertex. We can now state the road colouring theorem.
\begin{theorem}[Trahtman, \cite{Trahtman}]
Every strongly connected aperiodic graph with constant out-degree admits a synchronizing road colouring.
\end{theorem}
Note that the constant out-degree assumption is required for $G$ to be road-colourable (ie. for $G$ to admit any road colouring at all). To generalize road colourings to graphs with non-constant out degree, we introduce the notion of right resolving graph homomorphisms. We also introduce left and bi-resolving homomorphisms, which will be important in later sections.
\begin{definition}[Left resolvers, right resolvers, and bi-resolvers]
Let $\Phi:G\to H$ be a surjective graph homomorphism with edge map $\Phi:E(G)\to E(H)$ and vertex map $\partial\Phi: V(G)\to V(H)$. We say $\Phi$ is right resolving if the restriction $\Phi|_{E_I(G)}:E_I(G)\to E_{\partial\Phi(I)}(H)$ is a bijection for all $I\in V(G)$. We say $\Phi$ is left resolving if the restriction $\Phi|_{E^I(G)}:E^I(G)\to E^{\partial\Phi(I)}(H)$ is a bijection for all $I\in V(G)$. We say $\Phi$ is bi-resolving if it is both left and right resolving.
\end{definition}
 If there is a right resolver from graphs $G$ to $H$, we say that $H$ is a right resolving factor of $G$ and write $H\leq_R G$. The set of right resolvers is closed under composition, and if $H\leq_R G$ and $G\leq_R H$ then $H\cong G$ (a right resolver between graphs of the same size is an isomorphism), so the relation $\leq_R$ is a partial ordering on isomorphism classes of directed graphs. Ashley-Marcus-Tuncel proved the following uniqueness property under this order.
\begin{theorem}[Ashley-Marcus-Tuncel, \cite{AMT}]
\label{M(G)}
For any graph $G$, there is a unique $\leq_R$-minimal graph $M(G)\leq_R G$. Moreover, there is a unique vertex map $\Sigma_G:G\to M(G)$ such that $\partial\Phi=\Sigma_G$ for all right resolvers $\Phi:G\to M(G)$.
\end{theorem}
Now consider again the case of a graph $G$ with some constant out-degree $k$. A road colouring of $G$ corresponds to a right resolving homomorphism from $G$ to the graph with one vertex and $k$ self loops, which we call $M_k$. Such a homomorphism ``colours" $G$ with the self loops of $M_k$. Since $M_k$ is clearly $\leq_R$-minimal, we have $M(G)=M_k$ in this case.

In the same way that a sequence of colours can be followed from each vertex in  graph under a road colouring, a transition map on the vertices of a graph can be defined by ``lifting" paths through a right resolver from the range graph to the domain graph.
\begin{definition}
Let $\Phi:G\to H$ be a right resolver. For $I\in V(G)$ and $w\in L_{\partial\Phi(I)}(H)$, we write $I\cdot_\Phi w$ to denote the terminal vertex of the unique path $\pi\in L_I(G)$ with $\Phi(\pi)=w$. For $u\in L^{\partial\Phi(I)}(H)$, we write $u\cdot_\Phi I=\{J\in \partial\Phi^{-1}(i(u)):J\cdot_\Phi u=I\}$ for the set of vertices that lift to $I$ under $u$.
\end{definition}
Since a path $w$ in the range graph of a right resolver $\Phi: G\to H$ can only be lifted to paths that start in $\partial\Phi^{-1}(i(w))$, we define the synchronization $\Phi$ in terms of an equivalence relation that refines the fibers of $\partial\Phi$.

\begin{definition}[Stability relation]
Let $G$ and $H$ be graphs and let $\Phi:G\to H$ be a right resolver. The stability relation of $\Phi$, denoted by $\sim_\Phi$, is defined as follows: for $I\in V(H)$ and $I_1',I_2'\in \partial\Phi^{-1}(I)$, we say $I_1'\sim_\Phi I_2'$ if and only if, for all $u\in L_I(H)$ there is a $v\in L_{t(u)}(H)$ such that $I_1'\cdot_\Phi uv=I_2'\cdot_\Phi uv$.
\end{definition}

\begin{definition}[Synchronizing right resolver]
A right resolver $\Phi:G \to H$ is synchronizing if the equivalence classes of $\sim_\Phi$ are the entire fibers of $\partial\Phi$.
\end{definition}

We could have equivalently defined a synchronizing right resolver as a right right resolver $\Phi:G \to H$ such that, for all $I\in V(H)$, there is a $u\in L_I(H)$ such that $|\partial\Phi^{-1}(I)\cdot_\Phi u|=1$. Since the only fiber of a right resolver $\Phi: G\to M_k$ is all of $V(G)$, this clearly generalizes the notion of a synchronizing road colouring. The stability relation, however, will be crucially important in the inductive strategy used in later sections.

It will also be useful to extend the stability relation of a right resolver $\Phi:G\to H$ to the edges of $G$: for $e,f\in E(G)$, say that $e\sim_\Phi f$ if and only if $i(e)\sim_\Phi i(f)$ and $\Phi(e)=\Phi(f)$. Note also that if $I\sim_\Phi J$ for some right resolver $\Phi:G\to H$, then $I\cdot_\Phi e\sim_\Phi J\cdot_\Phi e$ for any $e\in E_{\partial\Phi(I)}(H)$. This congruence property allows us to define quotient graphs over stability relations.
\begin{definition}[Stability quotient]
Let $G$ and $H$ be graphs, and let $\Phi:G\to H$ be a right resolver. Define a quotient graph $G/\sim_\Phi$ by setting $V(G/\sim_\Phi)=V(G)/\sim_\Phi$ and $E(G/\sim_\Phi)=E(G)/\sim_\Phi$, where $i([e]_{\sim_\Phi})=[i(e)]_{\sim_\Phi}$ and $t([e]_{\sim_\Phi})=[t(e)]_{\sim_\Phi}$. 
\end{definition}

We will use the following results from \cite{Mac}, which describe the behaviour of synchronizers and stability quotients.

\begin{theorem}[MacDonald, \cite{Mac}]
Let $G$ and $H$ be graphs such that there is a right resolver $\Phi:G\to H$. Let $\Psi:G\to G/\sim_\Phi$ be the quotient map. Then $\Psi$ is a synchronizing right resolver, and there is a right resolver $\Delta:G/\sim_\Phi\to H$ such that $\Phi = \Delta\circ\Psi$. For any such $\Delta$, $\sim_\Delta$ is trivial.
\end{theorem}
\begin{theorem}[Macdonald \cite{Mac}, proposition 9.1.16 in \cite{LM} for strongly connected graphs]
\label{degMult}
Let $G$, $K$, and $H$ be graphs such that there are right resolvers $\Phi:G\to K$ and $\Psi:K\to H$. Then $\Psi\circ\Phi$ is synchronizing if and only if both $\Phi$ and $\Psi$ are synchronizing.
\end{theorem}

\section{The $O(G)$ and bunchy factor conjectures}

Similar to the partial ordering $\leq_R$, we can define a relation $\leq_S$ on graphs by saying $H\leq_S G$ if there is a synchronizing right resolver from $G$ to $H$. The relation $\leq_S$ is reflexive and antisymmetric on graphs up to isomorphism. Theorem \ref{degMult} shows that $\leq_S$ is also transitive, so $\leq_S$ is a partial ordering. The $O(G)$ conjecture can be formulated as a statement that parallels theorem \ref{M(G)}.
\begin{conjecture}[$O(G)$ conjecture]
For all strongly connected graphs $G$, there is a unique $\leq_S$-minimal graph $O(G)\leq_S G$.
\end{conjecture}
The implication from the $O(G)$ conjecture to the road colouring theorem follows by a result from \cite{AGW}, which shows that, for all strongly connected aperiodic graphs $G$ with constant out degree $k$, $G$ and $M_k$ have a common strongly connected synchronizing extension (ie. a strongly connected graph $K$ such that $G\leq_S K$ and $M_k\leq_S K$). Assuming the $O(G)$ conjecture, it follows that $O(G)=O(M_k)$. But $O(M_k)=M_k$, so $O(G)=M_k$, and in particular $M_k\leq_S G$. If we consider a synchronizing right resolver from $G$ to $M_k$ as a road colouring, we then get a synchronizing road colouring of $G$ on $k$ colours.

The related bunchy factor conjecture requires a definition of the class of bunchy graphs. We also define the classes of almost bunchy and weakly almost bunchy graphs, which will be relevant in later sections.
\begin{definition}[Bunchy, almost bunchy, and weakly almost bunchy graphs]
Let $G$ be a graph. We say $G$ is bunchy if $\Sigma_G|_{F(I')}:F(I')\to F(\Sigma_G(I'))$ is a bijection for all $I'\in V(G)$. We say $G$ is almost bunchy if, for all pairs of states $(I,J)\in V(M(G))\times V(M(G))$, there is at most one state $I'\in \Sigma^{-1}(I)$ such that $|\Sigma^{-1}(J)\cap F(I')|\geq 2$. We say $G$ is weakly almost bunchy if, for every pair of states $(I,J)\in V(M(G))\times V(M(G))$ such that $|E_I^J(M(G))|\geq 2$, there is at most one state $I'\in \Sigma^{-1}(I)$ such that $|\Sigma^{-1}(J)\cap F(I')|=|E_I^J(M(G))|$.
\end{definition}
Note that bunchy graphs are almost bunchy and almost bunchy graphs are weakly almost bunchy. MacDonald proved in \cite{Mac} that the $O(G)$ conjecture holds for almost bunchy (and therefore bunchy) graphs. We will also use three other properties of bunchy graphs proved in \cite{Mac}.
\begin{theorem}[MacDonald, \cite{Mac}]
\label{bunchy}
Let $G$ be a graph.
\begin{enumerate}
    \item The set of right resolving bunchy factors of $G$ has a unique $\leq_R$-maximal element $B(G)$.
    \item For any bunchy $H\leq_R G$ and right resolver $\Phi: G \to H$, there are right resolvers $\Psi:G\to B(G)$ and $\Delta: B(G)\to H$ such that $\Phi = \Delta\circ\Psi$.
    \label{5.10.2}
    \item $G$ has at most one $\leq_S$-minimal synchronizing bunchy factor.
    \label{5.9}
\end{enumerate}
\end{theorem}
The results of theoreom \ref{bunchy} lead us to the bunchy factor conjecture. Suppose it can be shown that every strongly connected graph has a bunchy synchronizing factor. Then every minimal synchronizing factor must be bunchy, and theorem \ref{bunchy} (\ref{5.9}) would show that every strongly connected graph has a unique synchronizing factor. We are thereby lead to the following conjecture.
\begin{conjecture}[Bunchy factor conjecture, version 1]
\label{BFC0}
Every strongly connected graph has a bunchy synchronizing factor.
\end{conjecture}
Note that, while the bunchy factor conjecture implies the $O(G)$ conjecture, if an individual graph $G$ has a bunchy synchronizing factor, it does not follow by theorem \ref{bunchy} (\ref{5.9}) that $G$ has a unique minimal synchronizing factor.

By theorem \ref{bunchy} (\ref{5.10.2}) and theorem \ref{degMult}, a graph $G$ has a bunchy synchronizing factor if and only if $B(G)\leq_S G$. The bunchy factor conjecture can therefore be equivalently stated as follows.
\begin{conjecture}[Bunchy factor conjecture, version 2]
\label{BFC}
For all strongly connected graphs $G$, $B(G)\leq_S G$.
\end{conjecture}
The bunchy factor conjecture was first posed by MacDonald in \cite{Mac}, where two more equivalent versions are also given.

It is not difficult to see that the only bunchy right resolving factor of a strongly connected aperiodic graph with constant out-degree $k$ is $M_k$. Thus $B(G)=M_k$, and so the statement of conjecture \ref{BFC} reduces to Trahtman's road colouring theorem when restricted to these graphs. The bunchy factor conjecture might also be provable using an inductive strategy similar to the one employed by Trahtman: by assuming that a strongly connected graph $G$ is not bunchy, try to find a right resolver $\Phi$ on $G$ with non-trivial stability congruence. We would then have $G/\sim_\Phi\leq_S G$ with $|V(G/\sim_\Phi)|<V(G)$. By induction on number of states, we get a bunchy graph $B$ such that $B\leq_S G/\sim_\Phi$ and therefore $B\leq_S G$. This would prove the first version of the bunchy factor conjecture. The next two sections prove the bunchy factor conjecture for two classes of graphs using a variation of this strategy.
\section{Proof of both conjectures for weakly almost bunchy graphs}
We now prove the $O(G)$ and bunchy factor conjectures for the class of weakly almost bunchy graphs. The proof mirrors the proof of the almost bunchy case given in \cite{Mac}. As in \cite{Mac}, we use the notion of a minimal image to find a non-trivial stability relation. 
\begin{definition}[Minimal image]
Let $G$ and $H$ be strongly connected graphs with a right resolver $\Phi: G\to H$. A minimal image is a set $U=\partial\Phi^{-1}(I)\cdot_\Phi u$ for some $I\in V(G)$ and $u\in L_I(H)$, such that for all $v\in L_{t(u)}(H)$, $|U\cdot_\Phi v|=|U|$.
\end{definition}
The following properties of minimal images are contained in \cite{Mac}, but for completeness we give a proof.
\begin{proposition}
\label{minImg}
Let $G$ and $H$ be strongly connected graphs with a right resolver $\Phi:G\to H$. Every minimal image of $\Phi$ is of the same size, and every vertex in $G$ is contained in some minimal image.
\end{proposition}
\begin{proof}
For the first assertion, let $U_1=\partial\Phi^{-1}(I_1)\cdot_\Phi u_1$ and $U_2= \partial\Phi^{-1}(I_2)\cdot_\Phi u_2$ be minimal images, where $I_i\in V(H)$ and $u_i\in L_{I_i}(H)$. By strong connectedness, there is a path $w$ from $t(u_1)$ to $I_2$ in $H$. Then $U_1\cdot_\Phi wu_2\subset U_2$, and by minimality of $U_1$, we have $|U_1\cdot_\Phi wu_2|=|U_1|$. Hence $|U_1|\leq|U_2|$. Similarly, $|U_2|\leq|U_1|$, so $|U_1|=|U_2|$.

For the second assertion, let $I'\in V(G)$ be arbitrary, let $U$ be a minimal image, and let $J'\in U$. By strong connectedness, there is a path $w$ from $J'$ to $I'$ in $G$. Then $U\cdot_\Phi \Phi(w)$ is a minimal image with $I'=J'\cdot_\Phi \Phi(w)\in U\cdot_\Phi \Phi(w)$.
\end{proof}
We will use the following sufficient condition for stability to find a right resolver with a non-trivial stability relation. This result is contained in \cite{Mac}, but was motivated by a similar result in the road colouring context due to Trahtman.
\begin{proposition}
\label{minImgStab}
Let $G$ and $H$ be strongly connected graphs with a right resolver $\Phi:G\to H$. If there are minimal images $U_1, U_2\subset\partial\Phi^{-1}(I)$ for some $I\in H$ such that $U_1\Delta U_2=\{J_1,J_2\}$, then $J_1\sim_\Phi J_2$.
\end{proposition}
We now show that we can find a non-trivial stability congruence for weakly almost bunchy graphs that are not bunchy, and then proceed by induction.
\begin{proposition}
\label{wabStab}
Let $G$ be a strongly connected weakly almost bunchy graph. If $G$ is not bunchy, then there is a right resolver $\Phi:G\to M(G)$ such that $\sim_\Phi$ is non-trivial.
\end{proposition}
\begin{proof}
Since $G$ is not bunchy, there are states $I, J\in V(M(G))$ such that there is a vertex $I_0'\in \Sigma_G^{-1}(I)$ with distinct follower states $J_1',J_2'\in\Sigma_G^{-1}(J)$ and edges $e_i\in E_{I_0'}^{J_i'}(G)$. If there is a state $I_1'\in\Sigma_G^{-1}(I)$ such that $|\Sigma_G^{-1}(J)\cap F(I_1')|=|E_I^J(M(G))|$, choose $I_0'=I_1'$. By weakly almost bunchiness, for each $I'\in \Sigma_G^{-1}(I)\setminus\{I_0'\}$, there is a pair of distinct edges $f_{I',1},f_{I',2}\in E_{I'}(G)$ with $t(f_{I',1})=t(f_{I',2})\in\Sigma_G^{-1}(J)$. Define a right-resolver $\Phi:G\to M(G)$ by setting $\Phi(f_{I',i})=\Phi(e_i)$ for all $I'\in \Sigma_G^{-1}(I)\setminus\{I_0'\}$, and making $\Phi$ otherwise arbitrary. By proposition \ref{minImg}, there is a minimal image $U\subset \Sigma_G^{-1}(I)$ that contains $I_0'$. Let $U_0=U\setminus\{I_0'\}$. Then $U_0\cdot_\Phi\Phi(e_1)=U_0\cdot_\Phi\Phi(e_2)$ since $I'\cdot_\Phi \Phi(e_1)=I'\cdot_\Phi \Phi(e_2)$ for all $I'\in \Sigma_G^{-1}(I)\setminus \{I_0\}$. By minimality of $U$, $J_i'\notin U_0\cdot_\Phi\Phi(e_i)$, so $(U\cdot_\Phi \Phi(e_i))\Delta( U\cdot_\Phi\Phi(e_1))=\{J_1',J_2'\}$. Therefore $J_1'\sim_\Phi J_2'$ by proposition \ref{minImgStab}.
\end{proof}

The following closure property of weakly almost bunchy graphs is used in the induction step. It is also used to prove the $O(G)$ conjecture for weakly almost bunchy graphs as a corollary of the bunchy factor conjecture. 
\begin{proposition}
\label{wabRRClose}
The class of weakly almost bunchy graphs is closed under right resolvers.
\end{proposition}
\begin{proof}
Let $G$ be weakly almost bunchy and let $\Phi:G\to H$ be a right-resolver. Let $I,J\in V(M)$ be such that $E_I^J(M(G))\geq 2$. Let $I'\in \Sigma_H^{-1}(I)$. By theorem \ref{M(G)} and the surjectivity of $\partial\Phi$, there is an $I''\in\Sigma^{-1}_G(I)$ such that $\partial\Phi(I'')=I'$. Let $J'\in \Sigma_H^{-1}(J)\cap F(I')$ and $e\in E_{I'}^{J'}$. Since $\Phi$ is right resolving, there is an $a\in E_{I''}(G)$ such that $\Phi(a)=e$. Then $t(a)\in \Sigma_G^{-1}(J)\cap F(I'')$ such that $\partial\Phi(t(a))=J'$. Hence $|\Sigma^{-1}_H(J)\cap F(I')|\leq|\Sigma^{-1}_G(J)\cap F(I'')|$. Since $G$ is weakly almost bunchy, it follows that $H$ is also weakly almost bunchy.

\end{proof}
The following two propositions show that the bunchy factor conjecture and the $O(G)$ conjecture are true for weakly almost bunchy graphs.
\begin{proposition}
\label{wabBF}
For all strongly connected weakly almost bunchy graphs $G$, $B(G)\leq_S G$.
\end{proposition}
\begin{proof}
If $|V(G)|=1$, the claim clearly holds. Suppose $|V(G)|=n>1$ and that the claim holds for graphs with fewer than $n$ vertices. If $G$ is bunchy, then $B(G)=G$ and so $B(G)\leq_S G$. Suppose $G$ is non-bunchy. By proposition \ref{wabStab}, there is a right resolver $\Phi:G\to M(G)$ such that $|V(G/\sim_\Phi)|<|V(G)|$. By proposition \ref{wabRRClose}, $G/\sim_\Phi$ is weakly almost bunchy. By the induction hypothesis, we then get $B(G/\sim_\Phi)\leq_S G/\sim_\Phi\leq_S G$. Let $\Psi: G\to B(G/\sim_\Phi)$ be a synchronizing right resolver. By theorem \ref{bunchy} (\ref{5.10.2}), there are right resolvers $\Delta: G\to B(G)$ and $\Theta: B(G)\to B(G/\sim_\Phi)$ such that $\Psi=\Theta\circ\Delta$. By theorem \ref{degMult}, $\Delta$ is synchronizing. Hence $B(G)\leq_S G$.
\end{proof}
\begin{proposition}
Every strongly connected weakly almost bunchy graph has a unique $\leq_S$-minimal synchronizing factor.
\end{proposition}
\begin{proof}
Let $G$ be strongly connected and weakly almost bunchy. Let $H\leq_S G$. By proposition \ref{wabRRClose}, $H$ is weakly almost bunchy. By proposition \ref{wabStab}, if $H$ is not bunchy then $H$ is not $\leq_S$-minimal. Hence every $\leq_S$-minimal sychronizing factor of $G$ is bunchy. Theorem \ref{bunchy} (\ref{5.9}) then shows that $G$ has a unique $\leq_S$-minimal synchronizing factor.
\end{proof}

\section{The bunchy factor conjecture for bi-resolving graphs}
J. Kari \cite{Kari} solved the road problem for Eulerian graphs a few years before Trahtman gave a proof of the general case. We now adapt Kari's proof to prove the bunchy factor conjecture for a class of graphs that generalizes Eulerian graphs. The key property used in Kari's proof is that Eulerian graphs with constant out-degree admit a road colouring for which no pair of states is synchronizable. A right resolver on a strongly connected graph has no synchronizable pair of states if and only if it is bi-resolving. It is therefore natural to generalize Kari's approach to prove the bunchy factor conjecture for graphs $G$ that admit a bi-resolving homomorphism onto $B(G)$. We will refer to these graphs themselves as bi-resolving. 

\begin{definition}
A graph $G$ is bi-resolving if there is a bi-resolving homomorphism $\Phi:G\to B(G)$.
\end{definition}
Note that this is weaker than the condition that $G$ admit a bi-resolver $\Phi:G \to M(G)$. If there is such a bi-resolver, by theorem \ref{bunchy} (\ref{5.10.2}) there are right resolvers $\Psi:G\to B(G)$ and $\Delta:B(G)\to M(G)$ such that $\Phi=\Delta\circ\Psi$. The factors $\Psi$ and $\Delta$ are then also bi-resolving.

The following will be used in proposition \ref{BRStab}. The converse holds similarly, but is not needed.

\begin{proposition}
\label{injIfSurj}
Let $G$ be a strongly connected graph, and let $\Phi:G\to H$ be right resolving. If $\Phi|_{E^I(G)}:E^I(G)\to E^{\partial\Phi(I)}(H)$ is surjective for all $I\in V(G)$, then it is injective for all $I\in V(G)$. Hence $\Phi$ is bi-resolving in this case.
\end{proposition}
\begin{proof}
Let $e_0\in E(H)$ be arbitrary. By strong connectedness, there is a path $e_1,\dots, e_n$ in $H$ from $t(e_0)$ to $i(e_0)$. Since $\Phi$ is right resolving and $\Phi|_{E^I(G)}$ is surjective for all $I\in\partial\Phi^{-1}(t(e_k))$, the lifting of each edge $e_k$ defines a surjection from $\partial\Phi^{-1}(i(e_k))$ to $\partial\Phi^{-1}(t(e_k))$. The composition of these surjections is then a surjection from $\partial\Phi^{-1}(i(e_0))$ to itself, and is therefore a bijection. The lifting of each edge $e_k$ must then define a bijection, so no two preimages of $e_0$ can have the same terminal vertex. Hence $\Phi|_{E^I(G)}$ is injective for all $I\in V(G)$. 
\end{proof}

We now aim to find, for any strongly connected bi-resolving graph $G$, a right resolver $\Phi:G\to B(G)$ with a non-trivial stability relation. As in Kari's proof, we use the notion of a maximal synchronized set.
\begin{definition}[Synchronized and maximal synchronized sets]
Let $\Phi:G\to H$ be right-resolving and let $G$ be strongly connected. A synchronized set is a set of the form $S=u\cdot_\Phi I'$ for some $I'\in V(G)$ and $u\in L^{\partial\Phi(I')}(H)$. A synchronized set $S\subset \partial\Phi^{-1}(J)$ for some $J\in V(H)$ is a maximal synchronized set (MSS) if $|v\cdot_\Phi S|\leq |S|$ for all $v\in L^{J}(H)$.
\end{definition}
If $S\subset \partial\Phi^{-1}(I)$ is a synchronized set under a right-resolver $\Phi:G\to H$, then so is $u\cdot_\Phi S$ for any $u\in L^I(H)$. The size of $u\cdot_\Phi S$ is also bounded above by $|V(G)|$, so there must be some $w\in L^I(H)$ such that, for all $v\in L^{i(v)}(H)$, $|vw\cdot_\Phi S|\leq |w\cdot_\Phi S|$. Therefore every right-resolver has an MSS contained in some fiber.
\begin{proposition}
\label{BRMSS}
Let $G$ be an strongly connected bi-resolving graph. Let $\Phi:G\to B(G)$ be a bi-resolver, and let $\Phi':G\to B(G)$ be a right resolver such that $\partial\Phi'=\partial\Phi$. Let $S\subset \partial\Phi'^{-1}(I)$ be a maximal synchronized set under $\Phi'$. Then $u\cdot_{\Phi'} S$ is an MSS under $\Phi'$ for any $u\in L^I(B(G))$. Consequently, every fiber of $\partial\Phi'$ contains an MSS under $\Phi'$.
\end{proposition}
\begin{proof}
Let $I'\in S\subset\partial\Phi'^{-1}(I)$. Since $\Phi$ is bi-resolving, $\Phi|_{E^{I'}(G)}:E^{I'}(G)\to E^I(B(G))$ is a bijection, and so $|E^{I'}(G)|=|E^I(B(G))|$. Since $\Phi'$ is right resolving, we then have
$$
\sum_{e\in E^{I}(B(G))} |e\cdot_{\Phi'} I'|= |E^{I'}(G)|=|E^I(B(G))|
$$
and using the right resolving property to sum over $S$ gives
\begin{align}
\sum_{e\in E^{I}(B(G))} |e\cdot_{\Phi'} S| = |E^I(B(G))||S|\label{eq1}
\end{align}
Since $S$ is an MSS, $|e\cdot_{\Phi'} S|\leq |S|$ for all $e\in E^I(B(G))$. By (\ref{eq1}), it must then be that $|e\cdot_{\Phi'} S| = |S|$ for all $e\in E^I(B(G))$. Therefore $e\cdot_{\Phi'} S$ is an MSS for all $e\in E^I(B(G))$ and so, by induction, $u\cdot_{\Phi'} S$ is an MSS for all $u\in L^I(B(G))$. 

By the strong connectedness of $B(G)$, for any $J\in V(B(G))$, there is a $v$ from $J$ to $I$. The set $v\cdot_{\Phi'} S$ is then an MSS contained in $\partial\Phi'^{-1}(J)$. Therefore every fiber of $\partial\Phi'$ contains an MSS.
\end{proof}
Suppose $\Phi: G \to H$ is a right resolver. For any $I\in V(H)$ and $w\in L_I(H)$, the equivalence classes of the relation on $\partial\Phi^{-1}(I)$ defined by synchronization by $w$ partitions $\partial\Phi^{-1}(I)$ into synchronized sets. The next proposition shows that, under some assumptions, $w$ can be chosen so that these synchronized sets are maximal.
\begin{proposition}
Let $G$ be a strongly connected bi-resolving graph. Let $\Phi:G\to B(G)$ be a right resolver for which there is a bi-resolver onto $B(G)$ with the same vertex map. For any $I\in V(B(G))$, there is a path $w\in L_I(B(G))$ that partitions $\partial\Phi^{-1}(I)$ into MSS's under $\Phi$.
\end{proposition}
\begin{proof}
Let $S_1,\dots, S_n\subset \partial\Phi^{-1}(I)$ be MSS's synchronized by a path $w$ to vertices $J_1',\dots, J_n'\in\partial\Phi^{-1}(t(w))$ respectively. By proposition \ref{BRMSS}, we can choose $w$ so that $n\geq 1$. Suppose there is an $I'\in\partial\Phi^{-1}(I)\setminus\bigcup_{i=1}^nS_i$. By strong connectedness of $G$, there is a path $u\in L_{t(w)}(B(G))$ such that $S_1\cdot_\Phi wu=I'$. Consider the MSS's synchronized by $wuw$. These include $wu\cdot_\Phi S_1,\dots,wu\cdot_\Phi S_n$, which are synchronized to $J_1',\dots, J_n'$ respectively, and $S_1$, which is synchronized to $I'\cdot_\Phi w$. Because these sets are synchronized to distinct vertices by the same path, they are disjoint. Repeating this process gives a path that partitions $\partial\Phi^{-1}(I)$ into maximal synchronized sets.
\end{proof}

The following is an illustration of the right resolvers $\Phi$ and $\Phi'$ in the proof of proposition \ref{BRStab}. For each of $\Phi$ and $\Phi'$, the left (resp. right) states in $G$ map to the left (resp. right) states in $B(G)$. The solid and dashed edges in $G$ map to the corresponding edges in $B(G)$.
\begin{center}
        \begin{tikzpicture}
            \tikzstyle{vtx} = [circle, minimum width = 2mm, fill, inner sep = 0pt]
            \node (G) at (-1.5,2) {$G:$};
            \node (M_G) at (-1.5,-1) {$B(G):$};
            
            \node (P) at (.75, 4.5) {$\Phi:$};
            
            \node[vtx] (I) at (0,2){};
            \node[vtx] (J_1) at (1.5,2.75){};
            \node[vtx] (J_2) at (1.5,1.25){};
            \node[vtx] (N_1) at (0,3.5){};
            \node[vtx] (N_2) at (0,.5){};
            
            \node (I_1) [right = 0cm of J_1]{$I_1'$};
            \node (I_2) [right = 0cm of J_2]{$I_2'$};
            
            \draw[- Stealth, thick, dashed] (N_1) to node[above]{$f_1$} (J_1);
            \draw[- Stealth, thick] (N_2) to node[below]{$f_2$} (J_2);
            \draw[- Stealth, thick] (I) to node[above]{$e_1$} (J_1);
            \draw[- Stealth, thick, dashed] (I) to node[below]{$e_2$} (J_2);
            
            \node[vtx] (M_1) at (0,-1){};
            \node[vtx] (M_2) at (1.5,-1){};
            
            \node (I) [right = 0cm of M_2] {$I$};
            
            \draw[- Stealth, thick] (M_1) to [bend left = -20] (M_2);
            \draw[- Stealth, thick, dashed] (M_1) to [bend left = 20] (M_2);

            \node (P') at (3.75, 4.5) {$\Phi':$};
            
            \node[vtx] (I') at (3,2){};
            \node[vtx] (J_1') at (4.5,2.75){};
            \node[vtx] (J_2') at (4.5,1.25){};
            \node[vtx] (N_1') at (3,3.5){};
            \node[vtx] (N_2') at (3,.5){};
            
            \node (I_1') [right = 0cm of J_1']{$I_1'$};
            \node (I_2') [right = 0cm of J_2']{$I_2'$};
            
            \draw[- Stealth, thick, dashed] (N_1') to node[above]{$f_1$} (J_1');
            \draw[- Stealth, thick] (N_2') to node[below]{$f_2$} (J_2');
            \draw[- Stealth, thick, dashed] (I') to node[above]{$e_1$} (J_1');
            \draw[- Stealth, thick] (I') to node[below]{$e_2$} (J_2');
            
            \node[vtx] (M_1') at (3,-1){};
            \node[vtx] (M_2') at (4.5,-1){};
            
            \node (I) [right = 0cm of M_2'] {$I$};
            
            \draw[- Stealth, thick] (M_1') to [bend left = -20] (M_2');
            \draw[- Stealth, thick, dashed] (M_1') to [bend left = 20] (M_2');
        \end{tikzpicture}
    \end{center}
    
\begin{proposition}
\label{BRStab}
Let $G$ be a non-bunchy, strongly connected graph such that there is a bi-resolving homomorphism $\Phi:G\to B(G)$. Then there is a right-resolving homomorphism $\Phi':G\to B(G)$ that has a non-trivial stability congruence $\sim_{\Phi'}$. Moreover, $\Phi'$ induces a bi-resolving homomorphism from $G/\sim_{\Phi'}$ onto $B(G)$. 
\end{proposition}
\begin{proof}
Since $G$ is non-bunchy, there are distinct vertices $I_1',I_2'\in V(G)$ and edges $e_i\in E^{I_i'}(G)$ such that $i(e_1)=i(e_2)$, and $\Sigma_G(I_1')=\Sigma_G(I_2')$. Since $B(G)$ is bunchy, we must then have $\partial\Phi(I_1')=\partial\Phi(I_2')=I$ for some $I\in V(B(G))$. Since $\Phi$ is bi-resolving, there are edges $f_i\in E^{I_i'}(G)$ such that $\Phi(f_1)=\Phi(e_2)$ and $\Phi(f_2)=\Phi(e_1)$. Define $\Phi':G\to B(G)$ by setting $\partial\Phi=\partial\Phi'$, $\Phi'(e_1)=\Phi(e_2)$, $\Phi'(e_2)=\Phi(e_1)$, and $\Phi'(e)=\Phi(e)$ for all $e\in E(G)\setminus\{e_1,e_2\}$. 

Let $S\subset \partial\Phi^{-1}(I)$ be such that $I_1'\in S$ and $I_2'\notin S$. Since $\Phi'$ is bi-resolving on every vertex but $I_1'$ and $I_2'$, $|\Phi'(e_1)\cdot_{\Phi'} (S\setminus\{I_1'\})|=|S|-1$. Also, $\Phi'(e_1)\cdot_{\Phi'} I_1'=\{i(e_1),i(f_1)\}$. Since $\Phi'$ is still right-resolving, we then have $|\Phi'(e_1)\cdot_{\Phi'} S|=|S|+1$. Therefore $S$ is not an MSS. Similarly, if $S\subset \partial\Phi^{-1}(I)$ such that $I_2'\in S$ but $I_1'\notin S$, then $S$ is not an MSS. Thus, in any partition of $\partial\Phi^{-1}(I)$ into MSS's, $I_1'$ and $I_2'$ must belong to the same partition element.

To show that $I_1'$ and $I_2'$ are stable under $\Phi'$, let $v\in L_I(B(G))$, and extend $v$ by strong connectedness to a cycle $v'$ from $I$ to $I$. Let $w\in F(I)$ be a path that partitions $\partial\Phi^{-1}(I)$ into some MSS's $S_1,\dots, S_n$. Then $v'w$ partitions $\partial\Phi^{-1}(I)$ into the MSS's $v'\cdot_{\Phi'} S_1, \dots, v'\cdot_{\Phi'} S_n$, and $\{I_1',I_2'\}\subset v'\cdot_{\Phi'} S_i$ for some $i\in\{1,\dots, n\}$, so $I_1'\cdot_{\Phi'} v'w=I_2' \cdot_{\Phi'} v'w$. Hence $I_1'\sim_{\Phi'} I_2'$.

We now show that the right-resolver from $G/\sim_{\Phi'}$ to $B(G)$ induced by $\Phi'$ is bi-resolving. Let $\Psi: G/\sim_{\Phi'}\to B(G)$ be the right-resolver given by $\partial\Psi([I']_{\sim_{\Phi'}}) = \partial\Phi(I')$ for $I'\in V(G)$ and $\Psi([e]_{\sim_{\Phi'}})= \Phi'(e)$ for $e\in E(G)$. By proposition \ref{injIfSurj}, it suffices to show that $\Psi|_{E^{[I']_{\sim_{\Phi'}}}}:E^{[I']_{\sim_{\Phi'}}}\to E^{\partial\Phi(I')}$ is surjective for all $[I']_{\sim_{\Phi'}}\in V(G/\sim_{\Phi'})$. Let $[I']_{\sim_{\Phi'}}\in V(G/\sim_{\Phi'})$, and let $I=\partial\Phi(I')$. If there is a $I_3'\in[I']_{\sim_{\Phi'}}\setminus\{I_1',I_2'\}$ then $\Phi'(E^{I_3'})=E^I$ so $\Psi(E^{[I']_{\sim_{\Phi'}}}) = E^I$. Otherwise, $[I']_{\sim_{\Phi'}}\subset\{I_1',I_2'\}$, and since $I_1'\sim_{\Phi'}I_2'$ we have $[I']_{\sim_{\Phi'}}=\{I_1',I_2'\}$. Since $\Phi'(E^{I_1'}\cup E^{I_2'})=E^I$, we still have $\Psi(E^{[I']_{\sim_{\Phi'}}})=E^I$. Hence $\Psi|_{E^{[I']_{\sim_{\Phi'}}}}$ is surjective.

\end{proof}
The following proposition shows that the bunchy factor conjecture holds for bi-resolving graphs.
\begin{proposition}
Let $G$ be a strongly connected bi-resolving graph. Then $B(G)\leq_S G$.
\end{proposition}
\begin{proof}
If $|V(G)|=1$, the claim clearly holds. Suppose $|V(G)|=n>1$ and that the claim holds for graphs with fewer than $n$ vertices. If $G$ is bunchy, then $B(G)=G$ and so $B(G)\leq_S G$. If $G$ is non-bunchy, by proposition \ref{BRStab}, there is a right-resolver $\Phi':G\to M(G)$ such that $|V(G/\sim_{\Phi'})|<|V(G)|$, and there is a bi-resolver from $G/\sim_{\Phi'}$ to $B(G)=B(G/\sim_{\Phi'})$. By the induction hypothesis, $B(G/\sim_{\Phi'})\leq_S G/\sim_{\Phi'}\leq_S G$, so $B(G)\leq_S G$.

\end{proof}
We have now shown that the bunchy factor conjecture holds for bi-resolving graphs, but we have not shown that the $O(G)$ conjecture holds similarly. In the case of weakly almost bunchy graphs, we used the fact that that class of graphs is closed under right resolvers to show that the bunchy factor conjecture implies the $O(G)$ conjecture. The class of bi-resolving graphs, however, is not closed under right resolvers, so a different or modified strategy would be needed.

\section{Empirical evidence for the bunchy factor conjecture}
We now present computer generated data that supports the bunchy factor conjecture, and by extension the $O(G)$ conjecture. In \cite{AMT}, a polynomial time algorithm is given for constructing $M(G)$, and deciding isomorphism between $M(G)$ and $M(H)$ for two graphs $G$ and $H$. In \cite{Mac}, polynomial time algorithms are given for constructing $B(G)$ and the stability relation of a right resolver. By implementing these algorithms we can estimate the probability that a given right resolver $\Phi: G\to B(G)$ is synchronizing. The bunchy factor conjecture states that this probability is nonzero for all strongly connected graphs. Not only did all of the graphs tested have a positive associated probability, but our results suggest that most right resolvers from a graph $G$ to its $B(G)$ are synchronizing.

For an irreducible graph $G$, the probability that a right resolver $\Phi:G\to B(G)$ is synchronizing was estimated by generating random right resolvers until 100 synchronizing right resolvers from $G$ to $B(G)$ were found (the fact that the testing procedure never failed to find a synchronizer shows that the bunchy factor conjecture is true for all graphs tested). The probability is then given by $p=100/\textit{\# of right resolvers generated}$. This estimator is derived from the fact that the total number of right resolvers generated is modeled by the sum of 100 geometric random variables, which has expectation $100/p$ where $p$ is the probability of success, in this case the probability of synchronization.

The following table records the average of the associated probability of graphs according to their minimal right resolving factor (given by an adjacency matrix) and number of states. Each entry in the table is an average over 10,000 graphs with the given $M(G)$ and $|V(G)|$.
\begin{center}
\begin{tabular}{|c|c|c|c|c|c|}
    \hline
    $M(G)$ / $|V(G)|$ & 4 & 5 & 6 & 7 & 8\\
    \hline
    $\begin{bmatrix}
    2&1\\
    1&0
    \end{bmatrix}$ & .983705 & .997026 & .996073 & .999273 & .999419\\
    \hline
    $\begin{bmatrix}
    1&2\\
    1&0
    \end{bmatrix}$ & .981192 & .982191 & .988474 & .993667 & .995971\\
    \hline
    $\begin{bmatrix}
    0&3\\
    1&0
    \end{bmatrix}$ & .945870 & .926934 & .938034 & .957786 & .975323\\
    \hline
\end{tabular}
\end{center}
The generally high probability (greater than $.9$) that a right resolver $\Phi:G \to B(G)$ is synchronizing is consistent with the constant out-degree case, where it is known that most road colourings are synchronizing. We might also expect that right resolvers are synchronizing with high probability (in a rigorous sense) since this was shown by Berlinkov in \cite{Berlinkov} to be true for road colourings. The small size of the graphs tested in our results, however, limits the insight we can provide on this point.

For most graphs, every right resolver from $G$ to $B(G)$ tested was synchronizing. The few graphs with an estimated associated probability of less than one often had a bimodal or multi-modal distribution. The following histograms illustrate these results. For each $M(G)$ and $|V(G)|$, the left histogram includes all graphs tested, while the right histogram includes only those with probability less than one.

\begin{center}
\begin{tikzpicture}
    \node (M) at (-1,0) {$M(G):$};
    \node[vtx] (g1) at (0,0){};
    \node[vtx] (g2) at (1,0){};
    
    \draw[-Stealth, thick] (g1) to [bend left] (g2);
    \draw[-Stealth, thick] (g2) to [bend left] (g1);
    
    \draw[Stealth-, thick] (g2)+(.7mm,-.9mm) arc (-160:160:2.5mm);
    \draw[Stealth-, thick] (g2)+(-.3mm,-1mm) arc (-165:165:4mm);
    
    \node at (-4,-3) {\includegraphics[scale=.33]{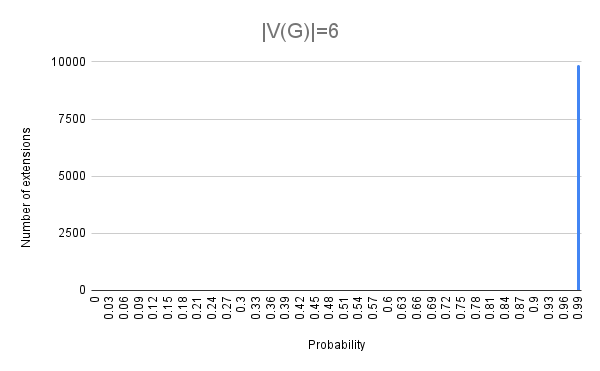}};
    \node at (4,-3) {\includegraphics[scale=.33]{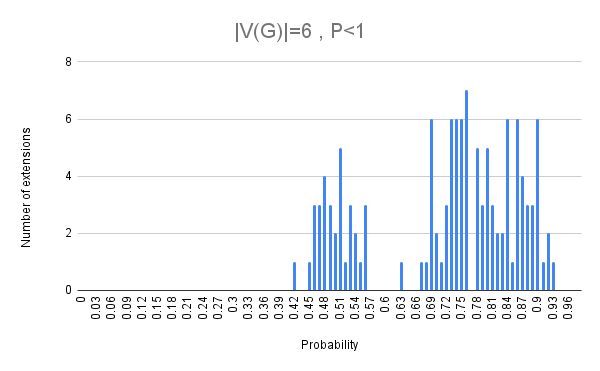}};
    
    \node at (-4,-8) {\includegraphics[scale=.33]{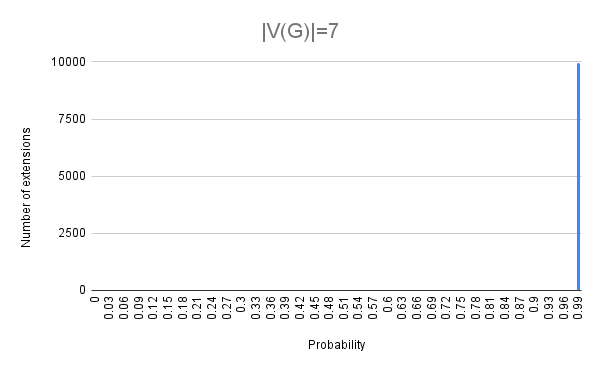}};
    \node at (4,-8) {\includegraphics[scale=.33]{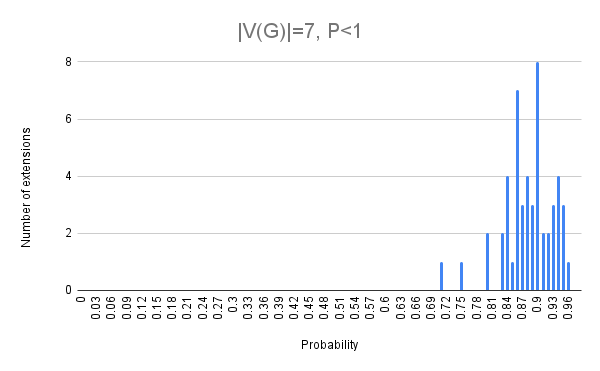}};
    
    \node at (-4,-13) {\includegraphics[scale=.33]{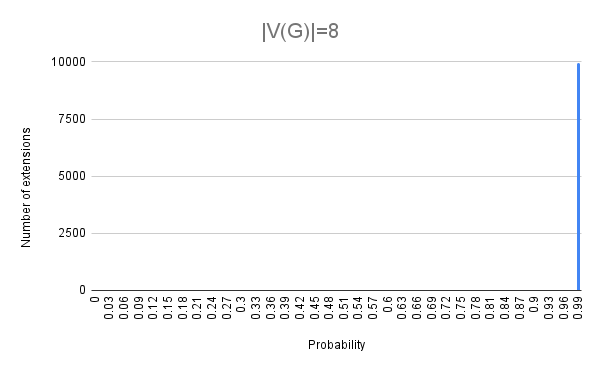}};
    \node at (4,-13) {\includegraphics[scale=.33]{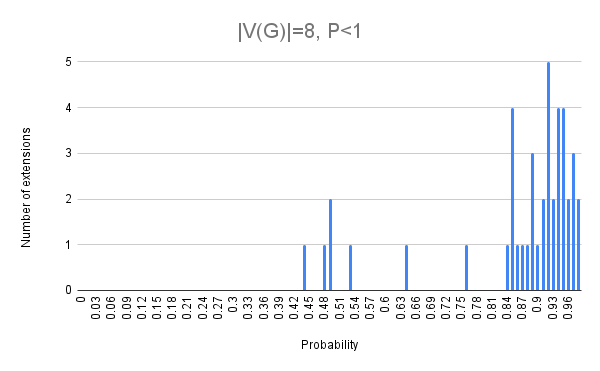}};
\end{tikzpicture}

\begin{tikzpicture}
    \node (M) at (-1,0) {$M(G):$};
    \node[vtx] (g1) at (0,0){};
    \node[vtx] (g2) at (1,0){};
    
    \draw[-Stealth, thick] (g1) to [bend left=35] (g2);
    \draw[-Stealth, thick] (g2) to [bend left = 0] (g1);
    \draw[-Stealth, thick] (g2) to [bend left = 35] (g1);
    
    \draw[Stealth-, thick] (g2)+(.4mm,-.8mm) arc (-163:163:3mm);
    
    \node at (-4,-3) {\includegraphics[scale=.33]{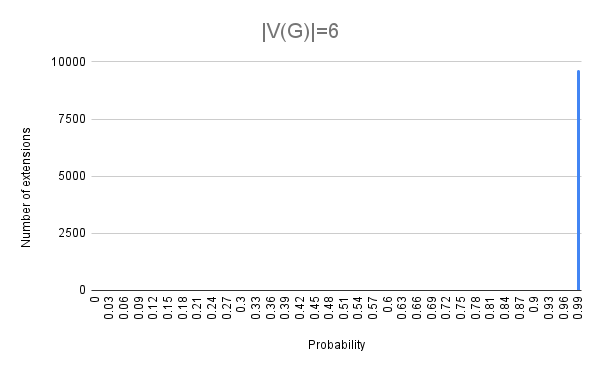}};
    \node at (4,-3) {\includegraphics[scale=.33]{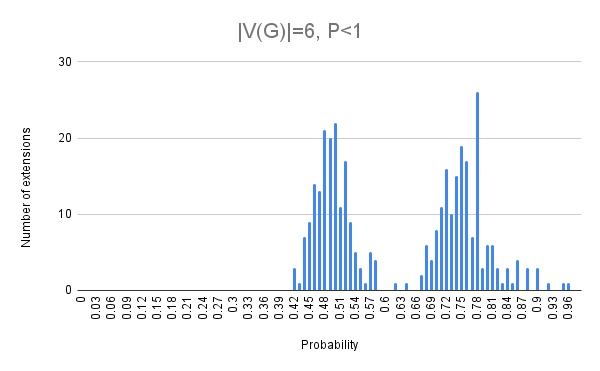}};
    
    \node at (-4,-8) {\includegraphics[scale=.33]{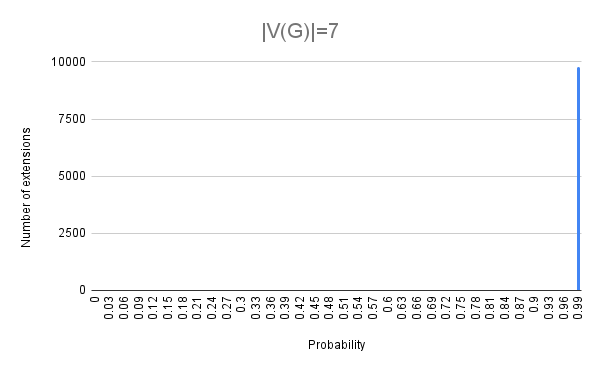}};
    \node at (4,-8) {\includegraphics[scale=.33]{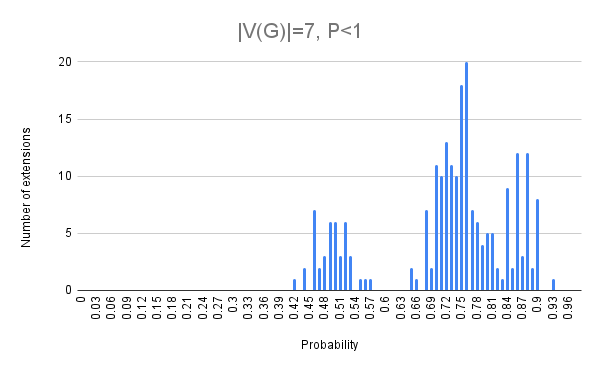}};
    
    \node at (-4,-13) {\includegraphics[scale=.33]{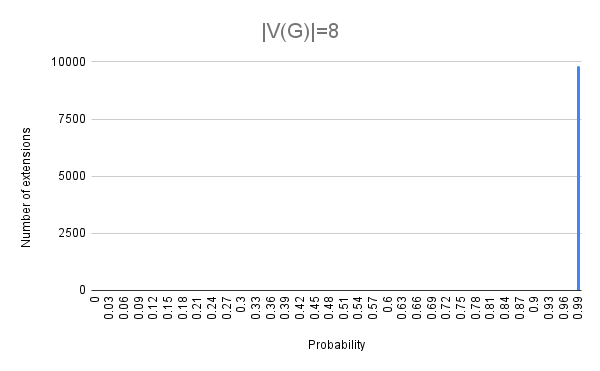}};
    \node at (4,-13) {\includegraphics[scale=.33]{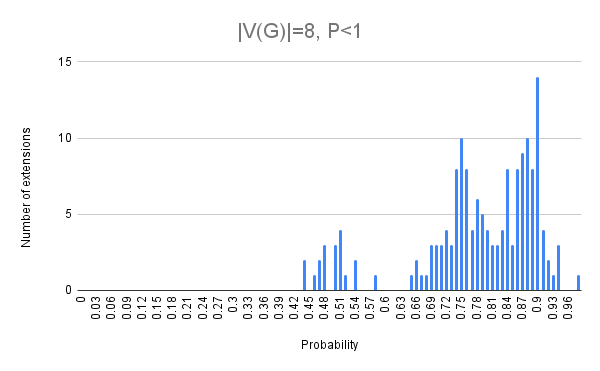}};
\end{tikzpicture}
\end{center}

\newpage

\section{Acknowledgments}
Many thanks to Sophie MacDonald and Brian Marcus for their generous feedback and support. The author also acknowledges the support of the Natural Sciences and Engineering Research Council of Canada (NSERC).

\bibliographystyle{plain}
\bibliography{refs}
\end{document}